\documentclass[11pt]{amsart}
\usepackage{amsmath, amssymb, latexsym, bm}

\renewcommand{\baselinestretch}{\baselinestretch}
\renewcommand{\baselinestretch}{1.1}
\numberwithin{equation}{section}

\newtheorem{thm}{Theorem}[section]

\newtheorem{lem}[thm]{Lemma}

\theoremstyle{definition}

\theoremstyle{remark}

\def\no[#1]{\| #1\|}
\def\dd[#1,#2]{\frac{#1}{#2}}
\def\peter[#1]{\langle #1\rangle}
\def\sno[#1]{\no[\mathcal{D}.#1]}

\newcommand{\Q}{\mathbb{Q}}
\newcommand{\R}{\mathbb{R}}
\newcommand{\Z}{\mathbb{Z}}

\newcommand{\GL}{\mathrm{GL}}

\newcommand{\diag}{\mathrm{diag}}

\newcommand{\ve}{\varepsilon}

\newcommand{\s}{\sigma}

\newcommand{\bv}{\boldsymbol{v}}

\newcommand{\bo}{\boldsymbol{0}}

\newcommand{\bxi}{\bm \xi}

\newcommand{\bx}{\bm x}

\newcommand{\bt}{\bm t}

\newcommand{\be}{\begin{equation}
}

\newcommand{\ee}{\end{equation}}

\begin{document}

\title[Explicit equivalence of quadratic forms]{EXPLICIT RESULT ON EQUIVALENCE OF RATIONAL QUADRATIC FORMS AVOIDING PRIMES}
\author{Wai Kiu Chan}
\address{Department of Mathematics and Computer Science, Wesleyan University, Middletown CT, 06459, USA}
\email{wkchan@wesleyan.edu}

\author{Haochen Gao}
\address{WesBox 91800, Wesleyan University, Middletown CT, 06459, USA}
\email{hgao@wesleyan.edu}

\author{Han Li}
\address{Department of Mathematics and Computer Science, Wesleyan University, Middletown CT, 06459, USA}
\email{hli03@wesleyan.edu}


\begin{abstract}
Given a pair of regular quadratic forms over $\Q$ which are in the same genus and a finite set of primes $P$, we show that there is an effective way to determine a rational equivalence between these two quadratic forms which are integral over every prime in $P$.  This answers one of the principal questions posed by Conway and Sloane in their book  {\em Sphere packings, lattices and groups},  Grundlehren der Mathematischen Wissenschaften [Fundamental Principles of Mathematical Sciences], Vol 290, Springer-Verlag, New York, 1999; page 402.
\end{abstract}
\maketitle


\section{Introduction}\label{intro}

A fundamental question in the arithmetic theory of quadratic forms is to decide when two given rational quadratic forms are integrally equivalent. For the sake of convenience, we will identify each quadratic form with its Gram matrix.   Given a pair of $n$-ary regular quadratic forms $F$ and $G$ over $\Q$,  the question is to decide whether there is a matrix $\tau \in \GL_n(\Z)$ such that
\begin{equation} \label{equivalence_problem}
\tau'F\tau = G,
\end{equation}
where $\tau'$ denotes the transpose of $\tau$.   Gauss' reduction theory provides a quite satisfactory solution to this question when $n = 2$.  Therefore in the subsequent discussion we will focus mainly on the case when $n \geq 3$, although many results mentioned later also hold for the binary case.  When $F$ and $G$ are positive definite, one can deduce from \eqref{equivalence_problem} explicit upper bounds on the height of $\tau$ in terms of the heights of $F$ and $G$ (the height of a matrix, denoted by $H$, and other height functions will be defined later in Section \ref{nota}).  Hence, in principle, we could perform an exhaustive search for $\tau$, and a fortiori provide an effective solution to the question in this case.

When $F$ and $G$ are indefinite, \eqref{equivalence_problem} has infinitely many solutions and an exhaustive search for $\tau$ is not possible.  However, Siegel \cite{S1} showed that there exists a function $C(n, F, G)$ such that if \eqref{equivalence_problem} has a solution then it must have one whose height is  less than $C(n, F, G)$.   Although Siegel did not give an explicit upper bound for $C(n,F,G)$, his method is effective.  Indeed, by following Siegel's argument Straumann \cite{ST} showed that
$$C(n, F, G) \leq \exp\left(k_n(H(F)H(G))^{\delta_n}\right)$$
where $\delta_n$ is an explicit polynomial in $n$ and $k_n$ is a constant depending only on $n$.  Masser \cite[page 252]{MA} conjectured that (for $n \geq 3$)
$$C(n,F,G) \ll_n (H(F) + H(G))^{\lambda_n}$$
where $\lambda_n$ is a constant depending only on $n$.  This conjecture has been confirmed by Dietmann \cite{D} for $n = 3$ (and for $n \geq 4$ when $\det(F)$ is cube-free), and by the third author and Margulis \cite[Theorem 1]{LM} for all $n \geq 3$ who also show that $\lambda_n$ can be taken to be a polynomial of $n$.  Thus, in principle, there is a deterministic algorithm to decide if $F$ and $G$ are equivalent over $\Z$ and, if they are, to exhibit an explicit integral equivalence $\tau$ satisfying \eqref{equivalence_problem}.

Another approach is appealing to the theory of spinor genus.  Before checking whether $F$ and $G$ are equivalent over $\Z$,  we could check first if they are in the same genus, that is, if they are equivalent over $\R$ and over $\Z_p$ for every prime $p$.  Over $\R$ this is straightforward; by Sylvester's Law of Inertia we just need to make sure that $F$ and $G$ have the same numbers (counted with multiplicity) of positive as well as negative eigenvalues.  Over $\Z_p$, this can be done effectively by using the invariants deriving from the Jordan decompositions of $F$ and $G$; see \cite[Chapter IX]{OM} or \cite[Chapter 8]{CA}.  When $n \geq 3$, the spinor genus and the class of an indefinite quadratic form coincide \cite[104:5]{OM}.  Therefore, the question is now reduced to deciding if $F$ and $G$ are in the same spinor genus, assuming that $F$ and $G$ are already in the same genus.  There are effective algorithms to do just that; see \cite{BH} or \cite[Chpater 13, Section 9]{CS}.  These algorithms usually require a rational matrix $\tau$ which satisfies \eqref{equivalence_problem} and is invertible over the ring $\Z^P := \bigcap_{p\in P} (\Z_p\cap \Q)$, where $P$ is the set of prime divisors of $2\det(F)$.  Finding such an explicit $\tau$ is one of the principal questions posed by Conway and Sloane in \cite[Page 402, Question (G4)]{CS}:
\begin{quote}
(G4)\quad ``{\em If two quadratic forms are in the same genus, find an explicit rational equivalence whose denominator is prime to any given number.}"
\end{quote}
Existence of such rational equivalence can be deduced from the weak approximation property  of the special orthogonal groups \cite[101:7]{OM}.  Siegel \cite{SI} also demonstrated the existence by a different argument making use of the Cayley Transformation. Both approaches have not been made effective.  However, a careful review of Siegel's argument suggests to us that his proof can be made effective and this is our approach to a solution to (G4). Our main result  (Theorem \ref{main_theorem}) is an explicit upper bound on the height of a skew-symmetric matrix whose image under the Cayley transformation is the rational equivalence wanted in (G4).

The rest of the paper is organized as follows. Throughout this paper, $k, m, n$ are positive integers and $p$ is always a prime number.  Section 2 contains preliminary materials on estimates pertaining to the $p$-adic valuations and on quadratic forms.  The main theorem, Theorem \ref{main_theorem}, will be presented in Section \ref{main}.

\section{Preliminaries}\label{nota}

For any prime number $p$, $\vert\,\,\vert_p$ is the $p$-adic valuation normalized so that $\vert p \vert_p = p^{-1}$.  The norm $\Vert\,\, \Vert_p$ on any $\Q_p^m$ is the $p$-adic sup-norm, that is,
$$\Vert \bxi \Vert_p = \max_{1\leq i \leq n}\{\vert \xi_i \vert_p \}, \quad \bxi = (\xi_1, \ldots, \xi_m) \in \Q_p^m.$$
Let $h_p(\bxi) : = \max\{\Vert \bxi \Vert_p, 1\}$, which is often called the $p$-adic inhomogeneous height of $\bxi$.  It is obvious that $\Vert \bxi \Vert_p \leq h_p(\bxi)$.  For any $m\times n$ matrix $A$ over $\Q_p$, $\Vert A \Vert_p$ and $h_p(A)$ are defined by viewing $A$ as a vector in $\Q_p^{mn}$.

\begin{lem} \label{INQ}
Let $X, Y$ be two $n\times n$ matrices over $\Q_p$.  Then:
\begin{enumerate}
        \item $\no[X + Y]_p \leq \max\{\no[X]_p, \no[Y]_p\}$.
        \item $\no[XY]_p \leq \no[X]_p\no[Y]_p$.
        \item $\vert \det(X)\vert_p \leq \no[X]_p^n \leq h_p(X)^{n}$.
        \item If $X$ is invertible, then $\no[X^{-1}]_p \leq \frac{\no[X]^{n-1}_p}{\vert\det(X)\vert_p}$.
\end{enumerate}
\end{lem}
\begin{proof}
This is clear.
\end{proof}

\begin{lem}\label{DET}
Let $X, Y$ be $n\times n$ matrices over $\Q_p$. If $\no[Y-X]_p < \frac{\vert\det(X)\vert_p}{h_p(X)^n}$, then $\vert\det(X)\vert_p = \vert\det(Y)\vert_p$.
\end{lem}
\begin{proof}
Let us write $Y = (y_{ij})$ and $X = (x_{ij})$, and let $S_n$ be the symmetric group on $\{1, \ldots, n\}$.  Notice that
\[
        \begin{aligned}
            \left\vert \det(Y) - \det(X)\right\vert_p &= \left\vert\sum_{\s \in S_n}\left( \mathrm{sgn}(\s) \prod_{i = 1}^{n} y_{i\s(i)}\right) - \sum_{\s \in S_n}\left( \mathrm{sgn}(\s) \prod_{i = 1}^{n} x_{i\s(i)}\right)\right\vert_p\\
                                  &= \left\vert\sum_{\s \in S_n}\mathrm{sgn}(\s) \left(\prod_{i = 1}^{n} y_{i\s(i)} - \prod_{i = 1}^{n} x_{i\s(i)}\right)\right\vert_p.
        \end{aligned}
\]
To finish the proof,  it suffices, by the ultra-triangle inequality, to show that the $p$-adic valuation of each term in the sum is strictly less than $\vert \det(X) \vert_p$.  For the sake of brevity, we will only demonstrate the analysis for one term. The readers will find no trouble in carrying out the same argument for the other terms.

Let us consider the term $\prod_{i = 1}^{n} y_{ii} - \prod_{i = 1}^{n} x_{ii}$.  By writing $y_{ii} = x_{ii} + \delta_i$, we see that
\[\prod_{i = 1}^{n} y_{ii} - \prod_{i = 1}^{n} x_{ii} = \prod_{i = 1}^{n} (x_{ii} + \delta_i) - \prod_{i = 1}^{n} x_{ii}\]
which is a sum of $2^n-1$ terms, each being a product of $n$ numbers in $\Q_p$ whose $p$-adic valuation is smaller than
$$\Vert X \Vert_p^{n-k} \, \Vert Y - X \Vert_p^k$$
for some $1 \leq k \leq n$.   Note that $\Vert X - Y \Vert_p^k \leq \Vert X - Y \Vert_p$ for any $k \geq 1$ because $\Vert X - Y \Vert_p < 1$ as a result of Lemma \ref{INQ}(3).  Therefore,
$$\Vert X \Vert_p^{n-k}\Vert X - Y \Vert_p^k \leq \Vert X - Y \Vert_p\, h_p(X)^n < \vert \det(X) \vert_p$$
as claimed.
\end{proof}

\begin{lem}\label{WAM} Let $P$ be a finite set of primes, $d$ be a positive integer, and $0<\ve \leq 1$ be a real number.  Suppose that for every $p \in P$,  an $x_p\in \Q_{p}$ is given such that $dx_p \in \Z_p$.  Then, there exists  $z \in \Z$ such that for each $p \in P$,
\[
\left\vert\frac{z}{d}-x_p\right\vert_{p}<\ve \quad \mbox{ and } \quad 0 \leq z < \prod_{p \in P} p^{\ell_p}
\]
where $\ell_p = \left\lceil \log_{p}\left(\frac{d}{\ve}\right) \right\rceil+1$.
\end{lem}
\begin{proof}
Since $\Z$ is dense in each $\Z_{p}$, there exists $z_p \in \Z$ such that
\[
\left\vert dx_p - z_p\right\vert_{p}<\frac{\ve}{d}, \qquad \forall\, p \in P.
\]
By the Chinese Reminder Theorem, there exists an integer $z$ such that  for each $p \in P$,
\[
z \equiv z_p \mod{p^{\ell_p}}, \qquad \mbox{ and } \qquad 0 \leq z < \prod_{p\in P} p^{\ell_p}.
\]
Moreover, since $p^{\ell_p}\,\frac{\ve}{d} > 1$,
\[
\begin{aligned}
\vert z-dx_p\vert_{p} &= \vert z- z_p + z_p - dx_p\vert_{p}\\
                                                    &\leq \max\{\vert z-z_p\vert_{p}, \vert z_p - dx_p\vert_{p}\}\\
                                                    &< \frac{\ve}{d}.
 \end{aligned}
\]
Then, since $d\vert d \vert_p \geq 1$, $\left\vert\frac{z}{d}-x_p\right\vert_{p}<\ve$ as claimed.
\end{proof}

The following lemma is essentially \cite[Lemma 15]{SI} but we draw a different conclusion at the end of its proof.

\begin{lem}\label{Siegel1}
Let $A \in \GL_n(\Q_p)$.  There exists at least one diagonal matrix $E$ whose diagonal entries are $1$ or $-1$ such that
\[
\vert \det(A-E)\vert_p \geq \vert 2^n \cdot \det(A)\vert_p.
\]
\end{lem}
\begin{proof}
Let $D$ be the diagonal matrix with indeterminate diagonal entries $\lambda_1, \lambda_2, \cdots, \lambda_n$. The determinant $\det(A - D)$ is a linear function of any of the $\lambda_k (k = 1,2...,n)$ and the same holds for the function
\be\label{TLF}
T(\lambda_1, \lambda_2, \cdots, \lambda_{n}) := \sum_{\Lambda \in \mathcal{E}} \det(A - \Lambda D),
\ee
where $\mathcal E$ is the group of $n\times n$ diagonal matrices with $\pm 1$ as the diagonal entries.  If $D$ is replaced by $\Lambda D$, for any $\Lambda \in \mathcal{E}$, the function $T$ is not changed.  Consequently $T$ is an even function of any of the variables $\lambda_k (k = 1,2...,n)$.  This proves that $T$ is a constant.  By taking in particular $D = I_n$ and $\bo$ in \eqref{TLF}, we obtain
\[
\left\vert\sum_{\Lambda \in \mathcal{E}} \det(A-\Lambda)\right\vert_p = \left\vert2^{n} \cdot \det(A)\right\vert_p.
\]
Therefore, there must be at least one $E \in \mathcal E$ such that $\vert\det(A-E)\vert_p \geq \vert 2^n \cdot \det(A)\vert_p$.
\end{proof}

For any $\bxi \in \Q^m$, the homogeneous height of $\bxi$ is
$$H(\bxi): = \Vert \bxi \Vert \prod_{p} \Vert \bxi \Vert_p$$
where $\Vert \bxi \Vert$ is the sup-norm of $\bxi$.  The inhomogeneous height of $\bxi$ is defined as
$$h(\bxi): = H((1, \bxi)) = \Vert (\bxi, 1) \Vert \prod_{p} h_p(\bxi).$$
It is not hard to see that for any positive number $C$, there are only finitely many  $\bxi \in \Q^m$ such that $h(\bxi) \leq C$ (the Northcott Property).   For any $m\times n$ matrix $A$ over $\Q$, $H(A)$ and $h(A)$ are defined by viewing $A$ as a vector in $\Q^{mn}$.

Let $F$ and $G$ be two $n$-ary regular quadratic forms over $\Q$.  There are two matrices  determined by $F$ and $G$ that will be taken as input data in Theorem \ref{main_theorem}:
\begin{enumerate}

\item[(i)] A matrix $\Sigma \in \GL_n(\Q)$ such that $\Sigma'F\Sigma$ is diagonal.

The columns $\bx_1, \ldots, \bx_n$ of $\Sigma$ form an orthogonal basis of $\Q^n$ with respect to the bilinear form induced by $F$.  Finding such an orthogonal basis is fairly straightforward.  We first pick a vector $\bx_1$ such that $F(\bx_1) \neq 0$.  The second vector $\bx_2$ must be in the orthogonal complement of $\bx_1$ and hence it is in the solution space of the homogenous system $\bx'F\bx_1 = \bo$.  We may continue this process to find the other $\bx_i$.

\item[(ii)] A matrix $\sigma \in \GL_n(\Q)$ such that $\sigma' F \sigma = G$.

By (i), we may assume that $G$ is already diagonalized, say $G = \diag(b_1, \ldots, b_n)$.  Since $G$ is regular, each $b_i$ is nonzero.  Finding $\sigma$ is tantamount to finding an orthogonal basis $\bt_1, \ldots, \bt_n$ of $\Q^n$ with respect to the symmetric bilinear form induced by $F$ such that $F(\bt_i) = b_i$ for $1 \leq i \leq n$.  Let $F_{b_1}$ be the $(n+1)$-ary quadratic form $F(x_1, \ldots, x_n) - b_1x_{n+1}^2$.  A theorem of Masser \cite{M} shows that there must be a vector $\bt_1 \in \Q^n$ such that $F(\bt_1) = b_1$ and
$$h(\bt_1) \leq 3^{\frac{n+1}{2}}\, n^{n+1}\, H(F_{b_1})^{\frac{n+1}{2}}.$$
As we mentioned earlier, there are only finitely many vectors in $\Q^n$ whose inhomogeneous height satisfy this inequality.  This gives us an effective procedure to find $\bt_1$.

The second vector $\bt_2$ must be in the orthogonal complement of $\bt_1$, which is the solution space of the homogeneous system $\bx'F\bt_1 = \bo$.  Take a basis $\bv_2, \ldots, \bv_n$ of this subspace, and let $T$ be the matrix whose columns are these $n-1$ vectors.  We may then apply Masser's theorem to the quadratic form $T'FT$ and find an explicit search bound for the inhomogeneous height of $\bt_2$.  Once again we have an effective procedure to find $\bt_2$.  By continuing this process in the obvious manner we should be able to find the other $\bt_i$.
\end{enumerate}

As in the proof of Lemma \ref{Siegel1}, we let $\mathcal E$ be the group of diagonal matrices with diagonal entries $1$ or $-1$.

\begin{lem}\label{Siegel3}
Let $F$ and $G$ be $n$-ary regular rationally equivalent quadratic forms over $\Q$.  Let $\sigma, \Sigma\in \GL_n(\Q)$ be such that
$\s'F\s = G$ and $\Sigma'F\Sigma$ is diagonal.  Then,
\begin{enumerate}
\item For any $E \in \mathcal{E}$,
\[
(\Sigma E\Sigma^{-1}\s)'F(\Sigma E\Sigma^{-1}\s) = G.
\]

\item There exists a matrix $\tau \in \{\Sigma E\Sigma^{-1}\s : E \in \mathcal{E}\}$ satisfying the following properties.  Suppose that $F$ and $G$ are equivalent over $\Z_p$.   Then there exists a $\tau_{p} \in \GL_n(\Z_{p})$ such that $\tau_{p}'F\tau_{p} = G$, that $\det(\tau - \tau_{p}) \neq 0$, and that
\[
\no[(\tau - \tau_{p})^{-1}]_p \leq h_p(\tau)^{n-1}\cdot\max\left\{\frac{1}{\vert 2^{n}\vert_{p}}, \frac{1}{\vert \det(\s)\vert_{p}}\right\}.
\]
\end{enumerate}
\end{lem}
\begin{proof} Part (1) follows from direct computation, which we leave the detail to the readers. For part (2), let $\mathrm{Equiv}(F, G)$ denote the set of all prime numbers $p$ for which $F$ and $G$ are equivalent over $\Z_p$, and let $p_0$ be the smallest element in $\mathrm{Equiv}(F, G)$. Let $\s_0$ be any matrix in $\GL_n(\Z_{p_0})$ satisfying
\[
\s_0'F\s_0 = G.
\]
Applying Lemma \ref{Siegel1} to $A_{0}:= \Sigma^{-1}\s_0\s^{-1}\Sigma$, we get that there exists an $E_0 \in \mathcal{E}$ such that
\be\label{pp0}
\vert\det(A_{0} - E_{0})\vert_{p_0} \geq \vert 2^n \cdot \det(A_0)\vert_{p_0}.
\ee
We shall prove that $\tau := \Sigma E_0\Sigma^{-1}\s \in \GL_n(\Q)$ satisfies our lemma.

First, for $p=p_0$, we take $\tau_{p_0}=\sigma_0$. Since $\s_0 \in \GL_n(\Z_{p_0})$, we have $\vert\det(\sigma_0)\vert_{p_0}=1$.  This fact and \eqref{pp0} give us
\[
\begin{aligned}
            \vert\det(\tau - \s_0)\vert_{p_0} &= \vert\det(\Sigma(A_0-E_0)\Sigma^{-1}\s)\vert_{p_0}\\
                                        &= \vert\det(\Sigma)\vert_{p_0}\, \vert\det(A_0-E_0)\vert_{p_0}\, \vert\det(\Sigma^{-1})\vert_{p_0}\, \vert\det(\s)\vert_{p_0}\\
                                        &\geq \vert 2^n \cdot \det(A_0)\vert_{p_0}\, \vert\det(\s)\vert_{p_0}\\
                                        &= \vert 2^n \cdot \det(\s) \cdot \det(\s_0) \cdot \det(\s^{-1})\vert_{p_0} \\
                                        &= \vert 2^n\vert_{p_0}.
\end{aligned}
\]
Therefore, as $\Vert \sigma_0 \Vert_{p_0} = 1$,
\[
\begin{aligned}
            \no[(\tau - \s_0)^{-1}]_{p_0} &\leq \frac{1}{\vert\det(\tau - \sigma_0)\vert_{p_0}} \no[\tau - \sigma_0]_{p_0}^{n-1}\\
                                              &\leq \frac{h_{p_0}(\tau)^{n-1}}{\vert 2^{n}\vert_{p_0}}.
\end{aligned}
\]
This proves that $\tau_{p_0}=\sigma_0$ satisfies the lemma.

Next we consider the case for $p\in\mathrm{Equiv}(A, B)$ with $p>p_0$. Then, plainly, $p>2$. By \cite[Page 115]{CA} or \cite[\S 92]{OM}, $F$ is equivalent to a diagonal quadratic form over $\Z_p$.  Thus, there exists a $\Sigma_p \in \GL_n(\Z_{p})$ such that $F_p: = \Sigma_p'F\Sigma_p$ is diagonal.  Let $\s_p$ be any matrix in $\GL_n(\Z_p)$ which satisfies
\[
        \s_p'F\s_p = G.
\]
By applying Lemma \ref{Siegel1} to $A_{p} := \Sigma_p^{-1}\tau\s_p^{-1}\Sigma_p$,  we get that there exists an $E_p \in \mathcal{E}$ such that
\[
\vert \det(A_{p} - E_{p})\vert_{p} \geq \vert 2^n \cdot \det(A_p)\vert_{p} = \vert \det(A_p)\vert_p.
\]
Let $\tau_p := \Sigma_pE_p\Sigma_p^{-1}\s_p$. Clearly, $\tau_p\in\GL_n(\Z_p)$. Since $E_p'F_pE_p = F_p$, we have $\tau_p'F\tau_p = G$. Notice that
\[
\begin{aligned}
            \vert\det(\tau - \tau_p)\vert_{p} &= \vert\det(\Sigma_p(A_p-E_p)\Sigma_p^{-1}\s_p)\vert_{p}\\
                                      &= \vert \det{\Sigma_p}\vert_{p}\, \vert\det(A_p-E_p)\vert_{p}\, \vert\det(\Sigma_p^{-1})\vert_{p}\, \vert \det(\s_p)\vert_{p}\\
                                      &\geq \vert \det(A_p)\vert_{p}\\
                                      &= \vert \det(\tau) \cdot \det(\s_p^{-1})\vert_{p}\\
                                      &= \vert \det(\s)\vert_{p}.
\end{aligned}
\]
 Hence we have
 \[
 \begin{aligned}
            \no[(\tau - \tau_p)^{-1}]_p &\leq \frac{1}{\vert \det(\tau - \tau_p)\vert_{p}} \no[(\tau - \tau_p)]_p^{n-1}\\
                                          &\leq \frac{h_p(\tau)^{n-1}}{\vert\det{\s}\vert_{p}}.
 \end{aligned}
 \]
 This implies that our choice of $\tau_p$ satisfies the lemma.
\end{proof}

Let $Q$ be an $n$-ary regular quadratic form over a field $K$ of characteristic $\neq 2$.  Let $\mathrm{O}_Q$ be the orthogonal group of $Q$ and $\textrm{Skew}_n$ be the set of $n\times n$ skew-symmetric matrices, both viewed as algebraic groups over $K$.  Let $L$ be an extension of $K$.  If $U \in \textrm{Skew}_n(L)$ such that $\det(U+Q) \neq 0$, then
\be\label{CT}
\mu := (U+Q)^{-1}(U-Q)
\ee
is an element in $\mathrm{O}_Q(L)$ with $\det(I_n-\mu) \neq 0$.  The map $U \longmapsto \mu$ is called the {\em Cayley Transformation} (or {\em Cayley-Dickson Parametrization}), which is a birational $K$-isomorphism from $\textrm{Skew}_n$ to $\mathrm{O}_Q$.   For any $\mu$ in $\mathrm{O}_Q(L)$ with $\det(I_n-\mu) \neq 0$, the inverse of this Cayley Transformation is defined and given by $\mu \longmapsto U = 2Q(I_n-\mu)^{-1}-Q$.  This $U$ is in $\textrm{Skew}_n(L)$ such that $\det(U+Q)\neq 0$ and \eqref{CT} holds.  The readers may consult \cite[Lemma 16]{SI} and \cite[Proposition 7.4]{PR} for more on these properties.

\section{ Main Theorem}\label{main}

Let $F$ and $G$ be two $n$-ary regular quadratic forms over $\Q$ which are equivalent over $\Q$.  Let $\Sigma$ and $\sigma$ be the rational matrices, as constructed in Section \ref{nota}, which have the properties that
\begin{equation} \label{k_sigma}
\Sigma' F \Sigma \mbox{ is diagonal } \qquad \mbox{ and } \qquad \sigma' F \sigma = G.
\end{equation}
Let $P$ be a finite set of primes.    For each $p \in P$, let
\begin{equation*} \label{kappap}
\kappa_p: = \max\{ h_p(\Sigma E \Sigma^{-1} \sigma) : E \in \mathcal E\},
\end{equation*}
where $\mathcal E$ is the group of diagonal matrices with diagonal entries $1$ or $-1$,
\begin{equation*} \label{alphap}
\alpha_p := \max\left\{\no[F]_{p}\,\kappa_p^{n}\,\max\left\{\frac{1}{\vert 2^{n}\vert_{p}}, \frac{1}{\vert \det{\s}\vert_{p}}\right\}, 1 \right\},
\end{equation*}
\begin{equation*} \label{betap}
\beta_p := \frac{\vert 2^n \det(F) \det(\s)\vert_{p}}{\kappa_p^{n}},
\end{equation*}
and
\be \label{ep}
\ve := \min_{p \in P}\left\{\frac{\beta_p}{\kappa_p\alpha_p^{n}}\right\}.
\ee
We define two positive constants, depending only on $n, \sigma, \Sigma, F$, and $P$, by
\be\label{dC}
d = \prod_{p \in P}\alpha_p, \qquad C = \prod_{p \in P} p^{\lceil \log_{p}\left(\frac{d}{\ve}\right)\rceil + 1}.
\ee
Recall that $\Z^P$ is the ring $\bigcap_{p\in P} (\Z_p\cap \Q)$.

\begin{thm} \label{main_theorem}
Let $P$ be a finite set of primes. Let $F$ and $G$ be $n$-ary regular quadratic forms over $\Q$.  Suppose that $F$ and $G$ are equivalent over $\Q$ and over  $\Z_{p}$ for each $p \in P$.   Let $\s, \Sigma \in \GL_n(\Q)$ be as in \eqref{k_sigma}, and $d, C>0$ be given by \eqref{dC}.  Then there exists a matrix
\[
\hat{\tau} \in \left\{(U+F)^{-1}(U-F)\Sigma E\Sigma^{-1}\s : E \in \mathcal{E}, U \in \frac{1}{d}\cdot\textnormal{Skew}_n\left(\Z\right), \no[U] \leq \frac{C}{d}\right\}
\]
such that
\[
\hat{\tau} \in \GL_n(\Z^{P}), \qquad \hat{\tau}'F\hat{\tau} = G.
\]

\end{thm}
\begin{proof}
Let $\tau \in \GL_n(\Q)$ and $\tau_{p} \in \GL_n(\Z_{p})$, $p \in P$,  be matrices satisfying Lemma \ref{Siegel3}.  We have $\tau = \Sigma E \Sigma^{-1}\sigma$ for some $E \in \mathcal E$ and
\be\label{TSTI1}
\tau' F\tau = G, \qquad \tau_{p}'F\tau_{p} = G, \qquad \tau - \tau_{p}\in \GL_n(\Q_p)
\ee
and
\be\label{TSTI2}
\no[(\tau - \tau_{p})^{-1}]_{p} \leq
        \kappa_p^{n-1}\max\left\{\frac{1}{\vert 2^{n}\vert_{p}}, \frac{1}{\vert\det{\s}\vert_{p}}\right\}.
\ee

A straightforward computation shows that $\tau_p\tau^{-1}$ is in $\mathrm{O}_F(\Q_p)$.   Moreover, by \eqref{TSTI1}, $\det(I_n - \tau_{p}\tau^{-1}) \neq 0$.  Therefore, we may apply the inverse of the Cayley Transformation to $\tau_{p}\tau^{-1}$ and write   \be\label{CT1}
        \tau_{p} = (U_{p}+F)^{-1}(U_{p}-F) \tau,
\ee
where $U_p = 2F(I_n - \tau_p\tau^{-1})^{-1} - F \in \textrm{Skew}_n(\Q_{p})$.

The next step is to find a $U \in \textrm{Skew}_n(\Q)$ such that
\be\label{detu+a}
   U \in \frac{1}{d}\cdot\textrm{Skew}_n(\Z), \qquad \Vert U \Vert \leq \frac{C}{d}, \qquad      \det(U+F)\neq 0,
\ee
and that the matrix
\be\label{HATTAU}
        \hat{\tau}:= (U+F)^{-1}(U-F)\tau \in \GL_n(\Q)
\ee
satisfies
\be\label{TSTI3}
        \no[\hat{\tau} - \tau_p]_{p} \leq 1.
\ee
This implies that $\hat{\tau}\in\GL_n(\Z_{p})$ for all $p \in P$; hence $\hat{\tau} \in \GL_n(\Z^P)$.   Moreover, since $\hat{\tau}\tau^{-1} \in \textrm{O}_F(\Q)$ because it is the image of $U$ under the Cayley Transformation.  Thus, $\hat{\tau}' F \hat{\tau} = G$ and this will finish the proof of the theorem.

To obtain \eqref{detu+a}, we will apply Lemma \ref{DET} to $X = U_p + F$ and $Y = U + F$, and Lemma \ref{WAM} to $U_p, \ve$ and $d$.  First of all,  we obtain from (\ref{CT1}) that
\be\label{UI}
        U_p + F = F (\tau + \tau_p) (\tau - \tau_p)^{-1} + F = F((\tau + \tau_p) (\tau - \tau_p)^{-1} + I_n).
\ee
Since $\tau_p \in \GL_n{(\Z_{p})}$, we have $\no[\tau_p]_{p} = 1$.  Then, by applying (\ref{TSTI2}) to (\ref{UI}), we get that
\[
        \begin{aligned}
            \no[U_p+F]_{p} &\leq \no[F]_{p} \max\left\{\no[\tau + \tau_p]_{p} \no[(\tau - \tau_p)^{-1}]_{p}, 1\right\}\\
                         &\leq \no[F]_{p}\, h_p(\tau)\, \kappa_p^{n-1}\, \max\left\{\frac{1}{\vert 2^{n}\vert_{p}}, \frac{1}{\vert \det{\s}\vert_{p}}\right\}\\
                         &\leq \alpha_p,
        \end{aligned}
\]
and hence $h_p(U_p + F) \leq \alpha_p$ because $1 \leq \alpha_p$.  Note that for the second inequality above, we have used
\[
     h_p(\tau)\,\kappa_p^{n-1}\,\max\left\{\frac{1}{\vert 2^{n}\vert_{p}}, \frac{1}{\vert \det{\s}\vert_{p}}\right\}\geq 1\cdot 1\cdot \dd[1, \vert 2^n\vert_{p}]\geq 1.
\]
This also implies that $\alpha_p \geq \Vert F \Vert_p$.  Consequently,
\be \label{UP}
\Vert U_p \Vert_p \leq \max\{\Vert U_p + F \Vert_p, \Vert F \Vert_p\} \leq \alpha_p.
\ee

Since $\vert \det(\tau - \tau_p)\vert_p \leq h_p(\tau)^n\leq \kappa_p^n$,  we have
\[
        \begin{aligned}
            \vert \det{(U_p + F)}\vert_{p} &= \left\vert 2^n \det(F) \frac{\det(\tau)}{\det(\tau - \tau_p)}\right\vert_{p}\\
                                    &\geq \frac{\vert 2^n \det(F) \det(\s)\vert_{p}}{\kappa_p^n} = \beta_p.
        \end{aligned}
\]
As a result,
\be \label{e}
\frac{\beta_p}{\kappa_p\alpha_p^n} \leq \frac{\beta_p}{\alpha_p^n} \leq \frac{\vert \det(U_p + F)\vert_p}{h_p(U_p + F)^n} \leq 1.
\ee

Let $\ve$ be as defined in \eqref{ep}, which is $\leq 1$ by \eqref{e}.   It follows from \eqref{UP} that $dU_p \in \mathrm{Skew}_n(\Z_p)$ for all $p \in P$.  We may then apply Lemma \ref{WAM} entry-wise and obtain a $U \in \frac{1}{d}\textrm{Skew}_n(\Z)$ such that
\be \label{U}
\no[U] \leq \frac{C}{d} \qquad \mbox{ and } \qquad \Vert U - U_p\Vert_p < \frac{\beta_p}{\kappa_p\alpha_p^n}, \qquad \forall\, p \in P.
\ee
Then, by \eqref{e},
$$\Vert U - U_p\Vert_p < \frac{\beta_p}{\kappa_p\alpha_p^n} \leq \frac{\beta_p}{\alpha_p^n} \leq \frac{\vert \det(U_p + F)\vert_p}{h_p(U_p + F)^n}.$$
Thus, by Lemma \ref{DET}, $\vert \det(U + F) \vert_p = \vert \det(U_p + F) \vert_p \geq \beta_p$, implying that $\det(U + F) \neq 0$.  Together with \eqref{U}, we obtain \eqref{detu+a}.

Now we come to the proof of \eqref{TSTI3}.  It follows from \eqref{CT1} and \eqref{HATTAU} that
$$
\hat{\tau} - \tau_{p} = (U + F)^{-1}(U-U_{p})(\tau-\tau_{p}).
$$
Then,
\[
 \begin{aligned}
            \no[\hat{\tau} - \tau_{p}]_{p} &\leq \no[(U+F)^{-1}]_{p} \no[\tau-\tau_{p}]_{p}\no[U-U_{p}]_{p}\\
                                             &\leq \frac{\no[U+F]_{p}^{n-1}}{\vert\det(U+F)\vert_{p}} \, h_p(\tau)\,\no[U-U_{p}]_{p}.
\end{aligned}
\]
Since $\Vert U_p + F \Vert_p \leq \alpha_p$ and $\Vert U - U_p\Vert_p < \frac{\beta_p}{\alpha_p^n} \leq 1 \leq \alpha_p$, we have $\Vert U + F \Vert_p \leq \alpha_p$.  Therefore,
\[
\begin{aligned}
            \no[\hat{\tau} - \tau_{p}]_{p} &\leq \frac{\alpha_p^{n-1}}{\beta_p}\, \kappa_p\, \no[U-U_{p}]_{p}\\
                                             & \leq \frac{\alpha_p^{n-1}}{\beta_p}\, \kappa_p\, \frac{\beta_p}{\kappa_p\alpha^n}\\
                                             & = \alpha_p^{-1}\\
                                             & \leq 1.
\end{aligned}
 \]
This finishes the proof of the theorem.
\end{proof}

\section*{Acknowledgment}

Haochen Gao would like to thank the support from Wesleyan Summer Research Program during the summers of 2019 and 2020 when part of this project was carried out. Han Li acknowledges support by the NSF Grant DMS 1700109.

\bibliographystyle{plain}

\end{document}